\documentclass[12pt]{article}
\usepackage[T1]{fontenc}
\usepackage{lmodern,amsmath,amsthm,amsfonts,amssymb,graphicx,float,microtype,thmtools,underscore,mathtools,xurl}
\usepackage[usenames,dvipsnames,svgnames,table]{xcolor}
\usepackage[shortlabels]{enumitem}
\setlist[itemize]{topsep=0ex,itemsep=0ex,parsep=0ex}
\setlist[enumerate]{topsep=0ex,itemsep=0ex,parsep=0ex}
\usepackage[unicode=true]{hyperref}
\hypersetup{ 
colorlinks,
breaklinks=true,
linkcolor={blue!60!black},
citecolor={black},
urlcolor={blue!60!black},
pdftitle={Product Structure Extension of the Alon--Seymour--Thomas Theorem}}
\usepackage[capitalise, compress, nameinlink, noabbrev]{cleveref}
\crefname{lem}{Lemma}{Lemmas}
\crefname{thm}{Theorem}{Theorems}
\crefname{cor}{Corollary}{Corlaries}
\newcommand{\defn}[1]{\textcolor{Maroon}{\emph{#1}}}

\usepackage[longnamesfirst,numbers,sort&compress]{natbib}
\makeatletter
\def\NAT@spacechar{~}
\makeatother
\usepackage[tmargin=25mm,bmargin=25mm,lmargin=30mm,rmargin=30mm]{geometry}
\renewcommand{\baselinestretch}{1.11}
\allowdisplaybreaks
\DeclarePairedDelimiter{\floor}{\lfloor}{\rfloor}
\DeclarePairedDelimiter{\ceil}{\lceil}{\rceil}

\DeclarePairedDelimiter{\set}{\{}{\}} 
\renewcommand{\epsilon}{\varepsilon}
\renewcommand{\emptyset}{\varnothing}
\renewcommand{\ge}{\geqslant}
\renewcommand{\le}{\leqslant}
\renewcommand{\geq}{\geqslant}
\renewcommand{\leq}{\leqslant}
\DeclareMathOperator{\dist}{dist}
\DeclareMathOperator{\tw}{tw}

\newcommand{\RR}{\mathbb{R}}

\newcommand{\PP}{\mathcal{P}}

\newcommand{\GG}{\mathcal{G}}

\newcommand{\LL}{\mathcal{L}}
\newcommand{\NN}{\mathbb{N}}
\newcommand{\OO}{\mathcal{O}}
\newcommand{\TT}{\mathcal{T}}
\newcommand{\david}[1]{{\color{orange} DW: #1}}

\newcommand{\michal}[1]{\textcolor{brown}{MS: #1}}

\newcommand{\torso}[2]{#1\langle #2\rangle}
\renewcommand{\thefootnote}{\fnsymbol{footnote}}
\theoremstyle{plain}
\newtheorem{thm}{Theorem}
\newtheorem{lem}[thm]{Lemma}
\newtheorem{cor}[thm]{Corollary}
\newtheorem{obs}[thm]{Observation}
\newtheorem{prop}{Proposition}
\newtheorem*{claim}{Claim}
\crefname{obs}{Observation}{Observations}
\newtheorem*{lem*}{Lemma}
\theoremstyle{definition}

\newtheorem*{conj*}{Conjecture}
\begin{document}
\title{\bf\boldmath\fontsize{18pt}{18pt}\selectfont
Product Structure Extension of the Alon--Seymour--Thomas Theorem}

\author{
Marc Distel\,\footnotemark[2] \qquad
Vida Dujmovi{\'c}\,\footnotemark[6]\qquad
David Eppstein\,\footnotemark[3] \\
Robert~Hickingbotham\,\footnotemark[2] \qquad
Gwena\"el Joret\,\footnotemark[4] \qquad
Piotr Micek\,\footnotemark[5] \\
Pat Morin\,\footnotemark[9]\qquad
Micha\l{} T.~Seweryn\,\footnotemark[4] \qquad
David~R.~Wood\,\footnotemark[2]
}

\maketitle

\begin{abstract}
Alon, Seymour and Thomas [1990] proved that every $n$-vertex graph excluding $K_t$ as a minor has treewidth less than $t^{3/2}\sqrt{n}$. Illingworth, Scott and Wood [2022] recently refined this result by showing that every such graph is a subgraph of some graph with treewidth $t-2$, where each vertex is blown up by a complete graph of order $\OO(\sqrt{tn})$. Solving an open problem of Illingworth, Scott and Wood [2022], we prove that the treewidth bound can be reduced to $4$ while keeping blowups of order $\OO_t(\sqrt{n})$. As an extension of the Lipton--Tarjan theorem, in the case of planar graphs, we show that the treewidth can be further reduced to $2$, which is best possible. We generalise this result for $K_{3,t}$-minor-free graphs, with blowups of order $\OO(t\sqrt{n})$. This setting includes graphs embeddable on any fixed surface.
\end{abstract}

\footnotetext[2]{School of Mathematics, Monash University, Melbourne, Australia (\texttt{\{marc.distel,robert.hickingbotham, david.wood\}@monash.edu}). Research of Wood supported by the Australian Research Council. Research of Distel and Hickingbotham supported by Australian Government Research Training Program Scholarships.}

\footnotetext[6]{School of Computer Science and Electrical Engineering, University of Ottawa, Ottawa, Canada (\texttt{vida.dujmovic@uottawa.ca}). Research supported by NSERC.}

\footnotetext[3]{Computer Science Department, University of California, Irvine (\texttt{eppstein@uci.edu}). Research supported in part by NSF grant CCF-2212129.}
 
\footnotetext[4]{D\'epartement d'Informatique, Universit\'e libre de Bruxelles, Belgium (\texttt{\{gwenael.joret,michal.seweryn\}@ulb.be}). G.\ Joret is supported by a CDR grant from the Belgian National Fund for Scientific Research (FNRS), by a PDR grant from FNRS, and by the Australian Research Council. Research of M.T.\ Seweryn supported by a PDR grant from the Belgian National Fund for Scientific Research (FNRS).}

\footnotetext[5]{Department of Theoretical Computer Science, Jagiellonian University, Kraków, Poland (\texttt{piotr.micek@uj.edu.pl}). Research supported 
the National Science Center of Poland under grant UMO-2018/31/G/ST1/03718 within the BEETHOVEN program.}

\footnotetext[9]{School of Computer Science, Carleton University, Ottawa, Canada (\texttt{morin@scs.carleton.ca}). Research supported by NSERC and the Ontario Ministry of Research and Innovation.}

\renewcommand{\thefootnote}{\arabic{footnote}}

\section{\Large Introduction}

Treewidth is a measure of how similar a given graph is to a tree, and is of fundamental importance in structural and algorithmic graph theory; see \citep{Bodlaender98,Reed03,HW17} for surveys. 

In one of the cornerstone results of Graph Minor Theory, \citet{AST90} proved that every $n$-vertex $K_t$-minor-free graph $G$ has treewidth $\tw(G)<t^{3/2}n^{1/2}$, which implies that $G$ has a balanced separator of order at most $t^{3/2}n^{1/2}$. For fixed $t\geq 5$, this bound is asymptotically tight since the $n^{1/2}\times n^{1/2}$ grid is $K_5$-minor-free and has treewidth~$n^{1/2}$. 

Our goal is to prove qualitative strengthenings of the Alon--Seymour--Thomas theorem through the lens of graph product structure theory, which describes graphs in complicated classes as subgraphs of products of simpler graphs. Here we consider products of bounded treewidth graphs and complete graphs. To be precise, for a graph $H$ and $m\in\NN$, let  $H\boxtimes K_m$ be the \defn{strong product} of $H$ and a complete graph $K_m$, which is the `complete-blow-up' of $H$ by $K_m$; that is, the graph obtained by replacing each vertex of $H$ by a copy of $K_m$ and replacing each edge of $H$ by the complete join between the corresponding copies of $K_m$. Say a graph $G$ is \defn{contained} in a graph $X$ if $G$ is isomorphic to a subgraph of $X$. 

\citet{ISW} showed that for any integer \(t \ge 4\), every $n$-vertex $K_t$-minor-free graph $G$ is contained in $H\boxtimes K_m$, for some graph $H$ with treewidth at most $t-1$, where $m < \sqrt{tn}$. This result implies and stregthens the Alon--Seymour--Thomas theorem since
\[\tw(G)\leq \tw(H\boxtimes K_m)\leq (\tw(H)+1)m-1 < t\sqrt{tn}.\]
Importantly, they also showed a similar result with treewidth \(t-2\) (and a slightly larger value of \(m\)): every $n$-vertex $K_t$-minor-free graph $G$ is contained in $H\boxtimes K_m$, for some graph $H$ with treewidth at most $t-2$, where $m < 2\sqrt{tn}$.

The following definition, implicitly introduced by \citet{ISW}, naturally arises. For a proper minor-closed graph class $\GG$, let \defn{$f(\GG)$} be the minimum integer such that for some $c$, every $n$-vertex graph $G\in \GG$ is contained in $H\boxtimes K_m$, for some graph $H$ with treewidth at most $f(\GG)$, where $m \leq c \sqrt{n}$.  
The above result of \citet{ISW} implies that $f(\GG)$ is well-defined; in particular, if $\GG_t$ is the class of $K_t$-minor-free graphs, then $f(\GG_t)\leq t-2$.
%, and if $\GG_{s,t}$ is the class of $K_{s,t}$-minor-free graphs, then $f(\GG_{s,t})\leq s$. 

\citet{ISW} asked whether $f(\GG)$ is upper bounded by an absolute constant. This paper answers this question in the affirmative.

\begin{thm}
\label{ConstantTwSqrtn}
Every $n$-vertex $K_t$-minor-free graph $G$ is contained in $H \boxtimes K_m$ for some graph $H$ of treewidth at most $4$, where $m\in O_t(\sqrt{n})$. 
\end{thm}

\cref{ConstantTwSqrtn} implies that $f(\GG)\leq 4$ for every proper minor-closed class $\GG$. The proof of \cref{ConstantTwSqrtn} actually shows that $\tw(H-v)\leq 3$ for some vertex $v\in V(H)$.

We also give improved bounds on $f(\GG)$ for particular minor-closed classes $\GG$. First consider the class $\LL$ of  planar graphs. The Lipton--Tarjan separator theorem~\citep{LT79} is one of the most important structural results about planar graphs, with numerous algorithmic applications~\citep{LT80}. It is equivalent to saying that every $n$-vertex planar graph has treewidth $\OO(\sqrt{n})$ (see \citep{DN19}). Since planar graphs are $K_5$-minor-free, the above result of \citet{ISW} shows that $f(\LL)\leq 3$. Our next contribution shows that $f(\LL)\leq 2$, resolving an open problem of \citet{ISW}. 

\begin{thm}
\label{Planar}
Every $n$-vertex planar graph is contained in $H\boxtimes K_m$, where $H$ is a graph with treewidth 2 and $m\in \OO(\sqrt{n})$.
\end{thm}

As an aside, since every graph with treewidth 2 is planar, the graph $H$ in \cref{Planar} is planar (although not necessarily a minor of the original planar graph). 

We actually prove a more general result than \cref{Planar} for graphs that exclude a $K_{3,t}$ minor. 

\begin{thm}
\label{K3t}
Every $K_{3,t}$-minor-free $n$-vertex graph is contained in $H\boxtimes K_m$, where $H$ is a graph with treewidth 2 and $m\in \OO(t\sqrt{n})$. 
\end{thm}

Since $K_{3,3}$ is not planar, \cref{K3t} with $t=3$ implies \cref{Planar}. More generally, \cref{K3t} also implies results for graphs embeddable in any fixed surface. The \defn{Euler genus} of a surface with~$h$ handles and~$c$ cross-caps is~${2h+c}$. The \defn{Euler genus} of a graph~$G$ is the minimum integer $g\geq 0$ such that there is an embedding of~$G$ in a surface of Euler genus~$g$; see \cite{MoharThom} for more about graph embeddings in surfaces. It follows from Euler's formula that $K_{3,2g+3}$ has Euler genus greater than $g$. Thus \cref{K3t} implies:

\begin{cor}
\label{Genus}
Every $n$-vertex graph with Euler genus $g$ is contained in $H\boxtimes K_m$, where $H$ is a graph with treewidth 2 and $m\in \OO((g+1)\sqrt{n})$.
%\michal{We can avoid `+1' by replacing "Euler genus \(g\)" with "Euler genus \(g > 0\)". The case \(g=0\) is handled by Theorem 2.} \david{I would rather keep this result a strict generalisation of the previous result.}
\end{cor}

Note that \citet{GHT-JAlg84} and \citet{Djidjev85a} proved that $n$-vertex graphs with Euler genus $g > 0$ admit balanced separators of order $\OO(\sqrt{gn})$ and thus have treewidth $\OO(\sqrt{gn})$. \cref{Genus} is a qualitative strengthening of these results, with slightly worse dependence on $g$. 

%\michal{We can avoid `qualitative strengthening' by simply `strengthening' if we replace $\OO(\sqrt{gn})$ by $\OO_g(\sqrt{n})$ (Since we hide the exact bound behind the big-O notation, why not go a step further and hide the dependency on \(g\)? Hiding the more precise bound from the reader makes our point clearer, and an interested reader can easily find the precise bound).} \david{I would rather make the dependence on $g$ explicit, especially since it is close to optimal.}

\subsection{Related Work} 

We first mention a connection to clustered colouring.
A (vertex-) $k$-colouring of a graph has \defn{clustering} $c$ if every monochromatic component has at most $c$ vertices. This is equivalent to saying that $G$ is contained in $H\boxtimes K_c$ for some graph $H$ with $\chi(H)\leq k$. Clustered colouring has been widely studied in recent years; see \citep{WoodSurvey} for a survey. \citet{LMST08} showed that $n$-vertex planar graphs, and more generally graphs excluding any fixed minor, are 3-colourable with clustering $\OO(\sqrt{n})$. Since treewidth 2 graphs are 3-colourable, in the case of planar or $K_{3,t}$-minor-free graphs, \cref{Planar,K3t} are a qualitative improvement over the result of \citet{LMST08}. 

Clustered colourings also provide lower bounds. \citet{LMST08} constructed a family of planar graphs $\{G_k:k\geq1\}$, where $G_k$ has $2k^3+1$ vertices and every $2$-colouring of~$G_k$ has a monochromatic component with at least $k^2/2$ vertices. 
In particular, if $G_k$ is contained in $H\boxtimes K_m$ for some graph $H$ with treewidth 1 (that is, $H$ is a forest), then a proper $2$-colouring of $H$ determines a 2-colouring of $G_k$ with clustering $m$, implying $m\in\Omega(n^{2/3})$ where $n:=|V(G_k)|$. Hence $f(\LL)>1$. Therefore the bounds on the treewidth of $H$ in \cref{Planar,K3t,Genus} are best possible. In particular,  $f(\LL)=2$, and if $\GG_{3,t}$ is the class of $K_{3,t}$-minor-free graphs, then $f(\GG_{3,t})=2$ for $t\geq 3$. These lower bounds lead to the following characterisation of minor-closed classes $\GG$ with $f(\GG) \leq 1$.
 
%More generally, $f(\GG) \leq 1$ for a proper minor-closed class $\GG$ if and only if $\GG$ has bounded treewidth, where the upper bound is proved by \citet[Theorem~8 with $t=1$]{DvoWoo}, and the lower bound follows from the Grid Minor Theorem~\citep{RS-V} together with the above-mentioned lower bound for planar graphs by \citet{LMST08}. 

\begin{prop}
For a minor-closed class $\GG$, $f(\GG) \leq 1$  if and only if $\GG$ has bounded treewidth.
\end{prop}
%\michal{How about we use an unnumbered environment since we do not reference this proposition later anyway?} \david{I don't like an unnumbered environment.}

\begin{proof}
\citet[Theorem~8 with $t=1$]{DvoWoo} proved that every $n$-vertex graph with treewidth $k$ is contained in $H\boxtimes K_m$ where $H$ is a star and $m\leq \sqrt{(k+1)n}$. Since a star has treewidth 1, if $\GG$ has bounded treewidth, then $f(\GG) \leq 1$. For the converse, if $\GG$ has unbounded treewidth, then by the Grid Minor Theorem~\citep{RS-V}, every planar graph is in $\GG$, and thus $f(\GG)\geq f(\LL)= 2$, as desired.  
\end{proof}

We conclude by mentioning the following related definition and results. \citet{UTW} defined the \defn{underlying treewidth} of a graph class $\GG$ to be the minimum integer $k$ such that for some function $g$ every graph $G\in\GG$ is contained in $H\boxtimes K_m$ where $\tw(H)\leq k$ and $m\leq g(\tw(G))$. Here $m$ is required to depend only on $\tw(G)$, whereas the present paper allows $m\in O(\sqrt{n})$. Amongst other results, \citet{UTW} showed\footnote{In the result of \citet{UTW}, $g(w)\in O_t(w^2\log w)$, which was improved to $O_t(w)$ by \citet{ISW}.} that the underlying treewidth of $\GG_t$ equals $t-2$. Thus, in the underlying treewidth setting, no absolute bound on $\tw(H)$ is possible, unlike in the setting of $\OO(\sqrt{n})$ blowups, where \cref{ConstantTwSqrtn} achieves $\tw(H)\leq 4$.  There is a similar distinction for planar graphs. \citet{UTW} showed that the underlying treewidth of the class of planar graphs equals 3. So in \cref{Planar} with $\tw(H)\leq 2$, the bound of $m\in \OO(\sqrt{n})$ cannot be improved to $m\leq g(\tw(G))$ for any function $g$. See \citep{DHHJLMMRW} for recent results on underlying treewidth. 

\section{\Large Background}
\label{Background}

For $m,n \in \mathbb{Z}$ with $m \leq n$, let $[m,n]:=\{m,m+1,\dots,n\}$ and $[n]:=[1,n]$. 

We consider simple, finite, undirected graphs~$G$ with vertex-set~${V(G)}$ and edge-set~${E(G)}$. 

For a graph $G$ and set $S\subseteq V(G)$, let $N_G(S):= \{v\in V(G)\setminus S: \exists\, vw\in E(G), w\in S\}$ and let $N_G[S]:= N_G(S)\cup S$. We drop the subscript $G$ if the graph in question is clear. 
%In a slight abuse of notation, for an induced subgraph $G'$ of $G$, let $N_{G'}(S):= N_G(S) \cap V(G')$ and $N_{G'}[S]:= N_{G'}(S)\cup S$. \gwen{This is inconsistent with the definition of $N_{G}(S)$ above, it should be avoided in my opinion. What about $N_{G, G'}(S)$?} \david{Where is the inconsistency? If $G'=G$ then $N_{G'}(S)=N_G(S)$ and $N_{G'}[S]=N_G[S]$. I also don't like the $N_{G, G'}(S)$ notation, since the reader will not understand what it means without consulting the definition. I have revised the paper to try and remove $N_{G'}(S)$ where $S\not\subseteq V(G')$. I think I have eliminated them all (please check), in which case we can remove the sentence above.} 

A \defn{tree-decomposition} of a graph $G$ is a collection $\TT=(B_x :x\in V(T))$ of subsets of $V(G)$ (called \defn{bags}) indexed by the vertices of a tree $T$, such that (a) for every edge $uv\in E(G)$, some bag $B_x$ contains both $u$ and $v$, and (b) for every vertex $v\in V(G)$, the set $\{x\in V(T):v\in B_x\}$ induces a non-empty (connected) subtree of $T$. The \defn{width} of~$\TT$ is $\max\{|B_x| \colon x\in V(T)\}-1$. The \defn{treewidth} of a graph $G$, denoted by \defn{$\tw(G)$}, is the minimum width of a tree-decomposition of $G$. 

Consider a tree-decomposition $\TT=(B_x :x\in V(T))$ of a graph $G$. The \defn{adhesion} of $\TT$ is $\max\{|B_x\cap B_y| \colon xy\in E(T)\}$. The \defn{torso} of a bag $B_x$ (with respect to $\TT$), denoted by \defn{$\torso{G}{B_x}$}, is the graph obtained from the induced subgraph $G[B_x]$ by adding edges so that $B_x\cap B_y$ is a clique for each edge $xy\in E(T)$. 
We say $\TT$ is \defn{rooted} if $T$ is rooted. Then, for each $x\in V(T)$, a clique $C$ in the torso $\torso{G}{B_x}$ is a \defn{child-adhesion clique} if there is a child $y$ of $x$ such that $C\subseteq B_x\cap B_y$. 

A \defn{path-decomposition} is a tree-decomposition in which the underlying tree is a path, simply denoted by the corresponding sequence of bags $(B_1,\dots,B_n)$.

A graph $H$ is a \defn{minor} of a graph $G$ if $H$ is isomorphic to a graph that can be obtained from a subgraph of $G$ by contracting edges. A graph~$G$ is \defn{$H$-minor-free} if~$H$ is not a minor of~$G$. A graph class $\GG$ is \defn{minor-closed} if every minor of every graph in $\GG$ is in $\GG$. A graph class is \defn{proper} if it is not the class of all graphs. 
The graph minor structure theorem of \citet{RS-XVI} shows that every $K_t$-minor-free graph has a tree-decomposition where each torso can be constructed using three ingredients: graphs on surfaces, vortices, and apex vertices. To describe this formally, we need the following definitions. 

Let $G_0$ be a graph embedded in a surface $\Sigma$. 
%A closed disc $D$ in $\Sigma$ is \defn{$G_0$-clean} if the interior of $D$ is disjoint from $G_0$, and the boundary of $D$ only intersects $G_0$ in vertices of $V(G_0)$ %\robert{I find the wording `only intersects $G_0$ in vertices of $V(G_0)$' difficult to understand. I suggest changing it to `intersect edges of $G_0$ only at their end-points.'} \david{I prefer `only intersects $G_0$ in vertices of $V(G_0)$', what do others think?} \eppstein{
A closed disc $D$ in $\Sigma$ is \defn{$G_0$-clean} if its only points of intersection with $G_0$ are vertices of $G_0$ that lie on the boundary of $D$. Let $x_1,\dots,x_b$ be the vertices of $G_0$ on the boundary of $D$ in the order around $D$. A \defn{$D$-vortex} (with respect to $G_0$) of a graph $H$ is a path-decomposition $(B_1,\dots,B_b)$ of $H$ such that $x_i\in B_i$ for each $i\in[b]$, and $V(G_0\cap H)=\{x_1,\dots,x_b\}$.

For integers $g,p,a\geq0$ and $k\geq1$, a graph $G$ is \defn{$(g,p,k,a)$-almost-embeddable} if for some set $A\subseteq V(G)$ with $|A|\leq a$, there are graphs $G_0,G_1,\dots,G_p$ such that:
\begin{itemize}
\item $G-A = G_{0} \cup G_{1} \cup \cdots \cup G_p$,
\item $G_{1}, \dots, G_p$ are pairwise vertex-disjoint,
\item $G_{0}$ is embedded in a surface $\Sigma$ of Euler genus at most $g$,
\item there are $p$ pairwise disjoint $G_0$-clean closed discs $D_1,\dots,D_p$ in $\Sigma$, and
\item for $i\in[p]$, there is a $D_i$-vortex $(B_1,\dots,B_{b_i})$ of $G_i$ of width at most $k$.
\end{itemize}
The vertices in $A$ are called \defn{apex} vertices---they can be adjacent to any vertex in $G$. A graph is \defn{$\ell$-almost-embeddable} if it is $(g,p,k,a)$-almost-embeddable for some $g,p,k,a\leq \ell$.

We use the following version of the graph minor structure theorem, which is implied by a result of \citet[Theorem~4]{DKMW12}.

\begin{thm}[\citep{DKMW12}]
\label{GMSTimproved}
For every integer $t\geq 1$ there exists an integer $k\geq 1$ such that every $K_t$-minor-free graph $G$ has a rooted tree decomposition $(B_x\colon x\in V(T))$ such that for every node $x\in V(T)$, the torso $\torso{G}{B_x}$ is $k$-almost-embeddable and if $A_x$ is the apex-set of $\torso{G}{B_x}$, then for every child-adhesion clique $C$ of $\torso{G}{B_x}$, either $C\setminus A_x$ is contained in a bag of a vortex of $\torso{G}{B_x}$, or $|C\setminus A_x|\leq 3$.
\end{thm}

The \defn{strong product} of graphs~$A$ and~$B$, denoted by~${A \boxtimes B}$, is the graph with vertex-set~${V(A) \times V(B)}$, where distinct vertices ${(v,x),(w,y) \in V(A) \times V(B)}$ are adjacent if
	${v=w}$ and ${xy \in E(B)}$, or
	${x=y}$ and ${vw \in E(A)}$, or
	${vw \in E(A)}$ and~${xy \in E(B)}$.

%An \defn{$H$-partition} of $G$ is a partition $\PP$ of $G$ such that $G/\PP$ is contained in $H$.  The following observation connects partitions and products.
%\david{We need to merge the previous and next paragraphs (from the two previous papers). There is an annoying issue with the above paragraph. For the observation below to be true we need to allow empty parts, and in fact we need to allow many empty parts. So $\PP$ should be defined to be a collection, not a set. But then the definition of quotient is problematic, since $V(G/\PP)$ needs to be  a set. We currently barely use the $G/\PP$ notation, so we could scrap it. The definition of $H$-partition below allows for multiple empty parts. We could define an $H$-partition $\PP$ of $G$ to be \defn{strict} if for all $x,y\in V(H)$, we have $xy\in E(H)$ if and only if there is an edge between $\PP_x$ and $\PP_y$ in $G$. Then in the proof of \cref{K3t} we have a strict $H$-partition, which would fix one of the remaining issues in a natural way. I am leaning towards scraping the $G/\PP$ notation, and just using $H$-partition and strict $H$-partition. Agreed?}

Let $G$ be a graph. A \defn{partition} of $G$ is a collection $\PP$ of sets of vertices in $G$ such that each vertex of $G$ is in exactly one element of $\PP$. Each element of $\PP$ is called a \defn{part}. Empty parts are allowed. The \defn{width} of $\PP$ is the maximum number of vertices in a part. The \defn{quotient} of $\PP$ (with respect to $G$) is the graph, denoted by \defn{$G/\PP$}, whose vertices are the non-empty parts in $\PP$, where distinct non-empty parts $A,B\in \PP$ are adjacent in $G/\PP$ if and only if some vertex in $A$ is adjacent in $G$ to some vertex in $B$. For a graph $H$, an \defn{$H$-partition} of $G$ is a partition $\PP=(\PP_x \subseteq V(G) :x\in V(H))$ of $G$ indexed by $V(H)$, such that for each edge $vw\in E(G)$, if $v\in \PP_x$ and $w\in \PP_y$ then $x=y$ or $xy\in E(H)$. That is, $G/\PP$ is contained in $H$. The following observation connects partitions and products.

\begin{obs}[\citep{DJMMUW20}]
\label{ObsPartitionProduct}
For all graphs $G$ and $H$ and any integer $p\geq 1$, $G$ is contained in $H\boxtimes K_p$ if and only if $G$ has an $H$-partition with width at most $p$.
\end{obs}

A \defn{layering} of a graph $G$ is a partition $\PP$ of $G$, whose parts are ordered $\PP=(V_0,V_1,\dots)$ such that for each edge $vw\in E(G)$, if $v\in V_i$ and $w\in V_j$ then $|i-j|\leq 1$. Equivalently, a layering is a $P$-partition for some path $P$. Consider a connected graph $G$. Let $r\in V(G)$ and let $V_i:=\{v\in V(G):\dist_G(v,r)=i\}$ for each $i\geq 0$. 
Then $(V_0,V_1,\dots)$ is a \defn{\textsc{bfs}-layering} of $G$ rooted at $r$. Let $T$ be a spanning tree of $G$, where for each non-root vertex $v\in V_i$ there is a unique edge $vw$ in $T$ for some $w\in V_{i-1}$. Then $T$ is called a \defn{\textsc{bfs}-spanning tree} of $G$. (These trees are a superset of the trees that can be generated by the breadth-first search algorithm.)\ 

If $T$ is a tree rooted at a vertex $r$, then a non-empty path $P$ in $T$ is \defn{vertical} if the vertex of $P$ closest to $r$ in $T$ is an end-vertex of $P$. 

Many recent results show that certain graphs can be described as subgraphs of the strong product of a graph with bounded treewidth and a path \citep{DJMMUW20,DHHW22,DMW23,HW24,DEMWW22,HJMW24,UWY22}. For example, \citet{DHHW22} proved the following result (building on the work of \citet{DJMMUW20}).

\begin{lem}[\citep{DHHW22}]
\label{GenusPartition}
Every connected graph $G$ of Euler genus at most $g$ is contained in $H\boxtimes P \boxtimes K_{\max\{2g,3\}}$ for some planar graph $H$ with treewidth 3, and for some path $P$. In particular, for every rooted spanning tree $T$ of $G$, there is a planar graph $H$ with treewidth at most $3$ and there is an $H$-partition $\PP$ of $G$ such that each part of $\PP$ is a subset of the union of at most $\max\{2g,3\}$ vertical paths in $T$.
\end{lem}

%%%%%%%%%%%%%%%%%%%%%%%%%%%%%%%%%%%%%%%%%
\section{\Large Proof of \cref{ConstantTwSqrtn}}
\label{TheProof}

This section proves \cref{ConstantTwSqrtn} for $K_t$-minor-free graphs, where the Graph Minor Structure Theorem is our main tool. We first prove an analogue of \cref{ConstantTwSqrtn} for almost-embeddable graphs with several additional properties that will be needed later. 

\begin{lem}\label{NewAlmostEmbeddable}
For integers $g,p,a\geq 0$ and $k,n\geq 1$ and $d\geq 4$, for every $(g,p,k,a)$-almost embeddable $n$-vertex graph $G$ with apex set $A$, there exists a set $S\subseteq V(G)$ where $|S|\leq \frac{n}{d-3}+a$ such that $G-S$ has an $H$-partition with width at most $(2g+4p+3)(2\sqrt{(k+1)n}+d+2k+2)$, where $H$ is planar with treewidth at most $3$. Moreover, $A\subseteq S$ and any clique in a vortex of $G$ is contained in at most two parts.
\end{lem}

\begin{proof}
Let $G_0,G_1,\dots,G_p$ and $D_1,\dots,D_p$ be as in the definition of $(g,p,k,a)$-almost embeddable. Let $G'_0$ be obtained from $G_0$ as follows. Initialise $G'_0:=G_0$ and add edges to $G'_0$ so that it is connected and is still embedded in the same surface as $G_0$, and $D_1,\dots,D_p$ are $G'_0$-clean.

For each $i\in [p]$, modify $G_0'$ as follows. Say
the vertices around $D_i$ are $x_1,\dots,x_b$. In $G_0'$, add edges so that $(x_1,\dots,x_b)$ is a path, and add a vertex $z_i$ into the disc $D_i$ adjacent to $x_1,\dots,x_b$. Note that since $D_i$ was initially $G_0'$-clean for each $i\in [p]$ and $D_1,\dots, D_p$ are pairwise disjoint, this can be done while maintaining an embedding of $G_0'$ in the same surface as $G_0$. 

Now apply the following operation for each $i\in [p]$. Let $(B_1,\dots,B_b)$ be a $D_i$-vortex of~$G_i$ with width at most $k$, where $x_j \in B_j$ for each $j\in[b]$. Greedily find an increasing sequence of integers $a_1,\dots,a_{q+1}$ so that $a_1=1$, $a_{q+1}=b+1$, and for each $j\in [q]$, if $Z_i:=B_{a_1}\cup B_{a_2} \cup \dots\cup B_{a_q}$ and $Y_{i,j}:=(B_{a_j+1}\cup B_{a_j+2}\cup\dots\cup B_{a_{j+1}-1})\setminus Z_i$ 
, then 
$\left\lceil\sqrt{(k+1)n}\right\rceil \leq |Y_{i,j}| \leq \left\lceil\sqrt{(k+1)n}\right\rceil+k$
for each $j\in[q-1]$ and
$|Y_{i,q}| \leq \left\lceil\sqrt{(k+1)n}\right\rceil+k$. Note that $n \geq (q-1) \sqrt{(k+1)n}$, so $|Z_i|\leq (k+1)q\leq
(k+1) ( n / \sqrt{(k+1)n} + 1) = \sqrt{(k+1)n} + k+1$. 

%\robert{Referee: Recursively defining the $a_j$ with ‘prior knowledge’ (as $Z_i$ depends on them) is a bit obscure. Give the reader more rope: e.g. ‘$a_1 = 1$. Choose $a_2$ minimal such that $Y_{i,1}:= (B_{a_1+1} \cup \dots \cup B_{a_2-1}) \ (B_{a_1} \cup B_{a_2})$ has size between $\sqrt{kn}$ and $\sqrt{kn+k}$. In general, choose $a_j+1$ minimal such that . . . ’. The fact that $a_j$ can always be chosen so that $Y_{i,j}$ has size within a range of length $k$ needed some thought (the bags have size at most $k + 1$) – you are using $B_{\ell} \cap B_{\ell+1} \neq \emptyset$ for all $\ell$. Spend slightly more time explaining the construction of the $Y_{i,j}$ and $Z_i$.}

Every clique in $G_i$ is contained in $Y_{i,j}\cup Z_i$ for some $j\in [q]$. In $G'_0$ contract the path $(x_{a_j+1},x_{a_j+2},\dots,x_{a_{j+1}-1})$ into a vertex $y_{i,j}$, for each $j\in [q]$. In $G'_0$ contract the edge $z_ix_{a_j}$ into $z_i$ for each $j\in [q]$. Call the vertices $y_{i,j}$ and $z_i$ of $G'_0$ \defn{special}. 

For each $i\in[p]$, let $F'_i$ be some face of $G'_0$ incident to $z_i$. If $p=0$ then add a vertex $r$ to $G'_0$ adjacent to some vertex of $G_0$. If $p\geq 1$ then for each $i\in[p-1]$, add a handle to the surface in which $G'_0$ is embedded between $F'_i$ and $F'_{i+1}$. The resulting embedding of $G'_0$ has a single face $F'$ incident to each of $z_1,\ldots,z_p$. Add a vertex $r$ to $G'_0$ adjacent to $z_1,\dots,z_p$. Embed $r$ and the edges incident to $r$ in $F'$. Note that (for any value of $p$) the resulting surface has Euler genus at most $g+2\max\{0,p-1\}\le g+2p$. 

Let $T$ be a \textsc{bfs}-spanning tree of $G'_0$ rooted at $r$ (which exists since $G'_0$ is connected). Let $(V_0,V_1,\dots)$ be the corresponding \textsc{bfs}-layering of $G'_0$. So $V_0=\{r\}$, and if \(p \ge 1\) then $V_1 = \{z_1, \ldots, z_p\}$ and $V_2$ contains all $y_{i,j}$ vertices (possibly plus others). By \cref{GenusPartition} there is an $H'$-partition $\PP'$ of $G'_0$ where $H'$ is planar with treewidth at most $3$ such that each part of $\PP'$ is a subset of the union of at most $\max\{2g+4p,3\}\leq 2g+4p+3$ vertical paths in $T$. Note that each vertical path in $T$ has at most two special vertices (some $z_i$ and some $y_{i,j}$). 
	 
For each $i\in [3,d-1]$, let $\widehat{V}_i:= V_i \cup V_{i+d}\cup V_{i+2d}\cup\cdots$. Since $|\widehat{V}_3|+|\widehat{V}_4|+\dots+|\widehat{V}_{d-1}| \leq n$, there exists $\ell\in [3,d-1]$ such that $|\widehat{V}_\ell|\leq n/(d-3)$. Let $S:=\widehat{V}_\ell\cup A$. Then $|S| \leq n/(d-3)+a$.
	
Let \(V_i := \emptyset\) for \(i < 0\), and for any integer $j\geq 0$, let $\PP_j'$ be the $H_j'$-partition of $G'_0[V_{\ell+(j-1)d+1}\cup \dots\cup V_{\ell+jd-1}]$ induced by $\PP'$, where $H_j'$ is a copy of $H'$ (and $H_0',H_1',\dots$ are pairwise disjoint). Then $\PP_j'$ has width at most $(2g+4p+3)d$.
	
Let $H$ be the disjoint union of $H_0', H_1', \dots$. Then $H$ is planar with treewidth at most $3$. Now $\PP_0'\cup\PP_1'\cup\dots$ is an $H$-partition of $G'_0-S$ where each part is a subset of the union of at most \((2g+4p+3)\) vertical paths of length at most \(d-1\) in \(T\). Hence, the width of this partition is smaller than $(2g+4p+3)d$.
	
We now modify this partition of $G'_0-S$ into a partition of $G-S$. By construction (since $\ell\geq 3$), $\PP_0'$ is a partition of $G'_0[V_0\cup V_1\cup V_2\cup\dots\cup V_{\ell-1}]$. In particular, each vertex $y_{i,j}$ (which is in $V_2$) is in some part $X$ of $\PP_0'$. Replace $y_{i,j}$ in $X$ by $Y_{i,j}$. Similarly, each vertex $z_{i}$ (which is in $V_1$) is in some part $X$ of $\PP_0'$. Replace $z_i$ in $X$ by $Z_i$. Remove $r$ from the part of $\PP_0'$ that contains $r$. This defines an $H$-partition $\PP$ of $G-S$ where every clique in a vortex of $G$ is contained in at most two parts. 
	
It remains to bound the width of $\PP$. Let $X\in \PP$. If $X$ comes from $\PP_j'$ for some $j\geq 1$, then $|X|\leq (2g+4p+3)d$. Now suppose $X$ comes from $\PP_0'$. Each vertical path in $T$ has at most two special vertices (some $z_i$ and some $y_{i,j}$). The corresponding replacements contribute at most \(2\sqrt{(k+1)n}+2k+2\) vertices to $X$. Since $X$ corresponds to the union of at most $2g+4p+3$ vertical paths (before replacement) in $T$, 
%$$|X|\leq (2g+4p+3)d + (2g+4p+3)(2\sqrt{kn}+2k) = (2g+4p+3)(2\sqrt{kn}+d+2k),$$
\begin{align*}
|X|
&\leq (2g+4p+3)d + (2g+4p+3)(2\sqrt{(k+1)n}+2k+2)\\
&= (2g+4p+3)(2\sqrt{(k+1)n}+d+2k + 2).    
\end{align*}
So $\PP$ has width at most $(2g+4p+3)(2\sqrt{(k+1)n}+d+2k + 2)$, as required. 
\end{proof}

To handle tree-decompositions we need the following standard separator lemma. For a tree $T$ rooted at $r\in V(T)$, the root of a subtree $T'$ of $T$ is the vertex in $V(T')$ that is closest to $r$. A \defn{weighted tree} is a tree $T$ together with a weighting function $\gamma:V(T)\to \RR^+$. The \defn{weight} of a subtree $T'$ of $T$ is $\sum_{v\in V(T')}\gamma(v)$. 

\begin{lem}
\label{TreeSep}
For every integer $q\geq 0$ and $n\in\RR^+$, every weighted tree $T$ with weight at most $n$ has a set $Z$ of at most $q$ vertices such that each component of $T-Z$ has weight at most $\frac{n}{q+1}$. 
\end{lem}

%\robert{This can be simplified. How about the following (let $q \geq 0$ in the statement so the base case is immediate). We proceed by induction on $q$. Root $T$ arbitrarily and for $u \in V (T)$, let  $T_u$ be the subtree of $T$ consisting of $u$ and all its descendants. Let $v$ be the lowest vertex in the tree such that $T_v$ has weight greater than $n/(q + 1)$ (if no such $v$ exists, then $T$ has weight at most $n/(q + 1)$ and so we may take $Z =\emptyset$ ). Now $T - V (T_v)$ has weight at most $qn/(q + 1)$ and so, by induction, has a set $Z′$ of at most $q - 1$ vertices whose deletion leaves components of weight at most $n/(q + 1)$. Let $Z = Z′ \cup  \{v\}$. Since $v$ was lowest, every component of $T_v - \{v\}$ has weight at most $n/(q + 1)$ and so all components of $T - Z$ have weight at most $n/(q + 1)$.} \david{done}

\begin{proof}
We proceed by induction on $q$. The $q=0$ case holds trivially with $Z=\emptyset$. Now assume that $q\geq 1$ and the result holds for $q-1$. Root $T$ at an arbitrary vertex $r$. For each vertex $v$, let $T_v$ be the maximal subtree of $T$ rooted at $v$. Let $v$ be a vertex in $T$ furthest from $r$ such that $T_v$ has weight greater than $\frac{n}{q+1}$.  (If no such $v$ exists, then $T$ has weight at most $\frac{n}{q+1}$ and $Z=\emptyset$ satisfies the claim). Let $T':=T-V(T_v)$. So $T'$ has weight at most $\frac{qn}{q+1}$. By induction, $T'$ has a set $Z'$ of at most $q-1$ vertices such that each component of $T'-Z'$ has weight at most $\frac{n}{q+1}$. Let $Z := Z' \cup  \{v\}$. By the choice of $v$, each component of $T_v-v$ has weight at most $\frac{n}{q+1}$. Thus each component of $T-Z$ has weight at most $\frac{n}{q+1}$.
\end{proof}

The next lemma handles tree-decompositions. 

\begin{lem}\label{CliqueSumNew}
    Let $a,b,k,n,w\geq 1$ be integers, and let $G$ be an $n$-vertex graph that has a rooted tree-decomposition $(B_x:x\in V(T))$ of adhesion at most $k$ such that for each $x\in V(T)$ there exists $S_x\subseteq B_x$ such that:
    \begin{itemize}
        \item $|S_x|\leq |B_x|/\sqrt{n}+a$ 
        \item $\torso{G}{B_x}-S_x$ has a $J_x$-partition $\PP_x$ of width at most $b$ where $\tw(J_x)\leq w$; and
        \item for every child-adhesion clique $C$ of $\torso{G}{B_x}$, the set $C\setminus S_x$ is contained in at most $w$ parts in $\PP_x$.
    \end{itemize}
     Then $G$ has an $H$-partition of width at most $\max\{b,(a+2k+1)\ceil{\sqrt{n}}\}$ such that $\tw(H)\leq w+1$. Moreover, $H$ contains a vertex $\alpha$ such that $\tw(H-\alpha)\leq w$.
\end{lem}

\begin{proof}
Let $r\in V(T)$ be the root of $T$. For every node $x\in V(T)$ with parent $y$, let $X_x:=B_x\cap B_y$ (where $X_r=\emptyset$) and let $B_x':=B_x-X_x$. For each node $x\in V(T)$, let $\gamma(x)=|B_{x}'|$. Observe that $({B}'_x \colon x \in V(T))$ is a partition of $V(G)$, so the total weight equals $n$. By \cref{TreeSep} with $q:=\ceil{\sqrt{n}}-1$, there is a set $Z'\subseteq V(T)$ where $|Z'|\leq q$ such that each component of $T-Z'$ has total weight at most $\frac{n}{q+1}\leq \sqrt{n}$. Let $Z:=Z'\cup \{r\}$. For each $z\in Z$, let $T_z$ be the maximal subtree of $T$ rooted at $z$ such that $T_z\cap Z=\{z\}$.
Let $Q:= \bigcup (X_{z}: z \in Z)$ and observe that $|Q|\leq k(q+1) \leq k\sqrt{n}$. For each $z \in Z$, let $G_z:=G[\bigcup(B_x': x \in V(T_z))]-X_z$. If there is an edge $xy\in E(T)$, where $y$ is the parent of $x$, and $x\in V(T_z)$ and $y\in V(T_{z'})$ for some distinct $z,z'\in Z$, then $x=z$. Thus $G-Q$ is the disjoint union of $(G_z\colon z\in Z)$.  

\begin{claim}
    For each $z \in Z$, there exists $S_z\subseteq V(G_z)$ where $|S_z|\leq |B_{z}|/\sqrt{n}+a$ such that $G_z-S_z$ has an $H_z$-partition of width at most $\max\{b,\sqrt{n}\}$ for some graph $H_z$ with treewidth at most $w$.
\end{claim}
  
\begin{proof}
    Let $T_1',\dots,T_f'$ be the components of $T_z-z$. For each $j\in [f]$, let $C_j$ be the subgraph of $G_z$ induced by $(B_x'\colon x\in V(T_j'))$. Note that $V(G_z)$ is the disjoint union of $B_z',V(C_1),\dots,V(C_f)$. By \cref{TreeSep}, $|V(C_j)|\leq |\bigcup(B_x' \colon x\in V(T_j'))|= \gamma(T_j')\leq \sqrt{n}$. By assumption, there is a set $S_z\subseteq B_{z}$ where $|S_z|\leq |B_{z}|/\sqrt{n}+a$ such that $\torso{G}{B_{z}}-S_z$ has a $J_z$-partition $\PP_z'$ with width at most $b$ where $\tw(J_z)\leq w$, and for every child-adhesion clique $C$ in $\torso{G}{B_{z}}$, $C \setminus S_{z}$ is contained in at most $w$ parts in $\PP_z'$. Let $(W_x^{(z)}\colon x\in V(T^{(z)}))$ be a tree-decomposition of $J_z$ with width at most $w$. Add $V(C_1),\dots, V(C_f)$ to the partition $\PP_z'$ to obtain a partition $\PP_z$ of $G_z-S_z$ with quotient $H_z$.     
    Then $\PP_z$ has width at most $\max\{b,\sqrt{n}\}$. For each $j\in [f]$, let $\alpha_j\in V(H_z)$ be the vertex that indexes $V(C_j)$ and let $N_j$ be the neighbourhood of $\alpha_j$. Since the neighbourhood of $C_j$ in $G_z$ is a child-adhesion clique of $\torso{G}{B_{z}}$, it follows that $N_j$ is a clique in $J_z$ of size at most~$w$. Thus there is a node $x\in V(T^{(z)})$ such that $N_j\subseteq W_x^{(z)}$. Add a leaf node $\ell$ adjacent to $x$ and let $W_{\ell}^{(z)}:=N_j\cup \{\alpha_j\}$. Repeat this procedure for all $j\in [f]$ to obtain a tree-decomposition of $H_z$ with width at most $w$.
\end{proof}
 Observe that 
 \[\sum_{z\in Z} |B_{z}|\leq 
 \sum_{z\in Z}(|B_z'|  + |X_z|) =
 \big(\sum_{z\in Z}|B_z'| \big) + \big(\sum_{z\in Z} |X_z| \big)
 \leq n+k|Z|.\] 
Since $|Q|\leq k|Z|$ and $|Z|\leq q+1= \ceil{\sqrt{n}}$,
\begin{align*}
    |Q\cup (\bigcup S_z:z\in Z)| 
     \leq |Q| + \sum_{z\in Z} \big(|B_{z}|/\sqrt{n}+a\big) 
&  \leq  (k+a)|Z| + (n+k|Z|)/\sqrt{n} \\
& <  (2k+a+1)\ceil{\sqrt{n}} .
     \end{align*}
Let $H$ be the graph obtained from the disjoint union of $(H_z \colon z\in Z)$ by adding one dominant vertex $\alpha$. So $\tw(H-\alpha)\leq w$ and $\tw(H)\leq w+1$. By associating $Q\cup (\bigcup S_z:z\in Z)$ with $\alpha$, we obtain an $H$-partition of $G$ with width at most $\max\{b,(a+2k+1)\ceil{\sqrt{n}}\}$.
\end{proof}

\begin{proof}[Proof of \cref{ConstantTwSqrtn}.]
Let $G$ be an $n$-vertex $K_t$-minor-free graph. By \cref{GMSTimproved}, $G$ has a rooted tree-decomposition $(B_x:x\in V(T))$, such that for each $x\in V(T)$, the torso $\torso{G}{B_x}$ is $k$-almost-embeddable (for some $k=k(t)$), and if $A_x$ is the apex-set of $\torso{G}{B_x}$, then for every child-adhesion clique $C$ of $\torso{G}{B_x}$, either $C\setminus A_x$ is contained in a vortex of $\torso{G}{B_x}$, or $|C\setminus A_x|\leq 3$. \citet[Lemma 21]{DMW17} showed that every clique in a $k$-almost-embeddable graph has at most $9k$ vertices. So the adhesion of $(B_x:x\in V(T))$ is at most $9k$. We may assume that $n\geq k$. By \cref{NewAlmostEmbeddable} with $d:=\ceil{\sqrt{n}}+3$, for each torso $\torso{G}{B_x}$ there exists a set $S_x\subseteq B_x$ such that $|S_x|\leq \frac{|B_x|}{\ceil{\sqrt{n}}}+k \leq \frac{|B_x|}{\sqrt{n}}+k$ and $\torso{G}{B_x}-S_x$ has a $J_x$-partition $\PP_x$, where $\tw(J_x)\leq 3$
and the width of \(\PP_x\) is at most $(2k+4k+3)(2\sqrt{(k+1)n}+\ceil{\sqrt{n}}+3+2k+2) \leq (6k+3)\cdot9\sqrt{(k+1)n}$ (because \(n \ge k \ge 1\)). 

Moreover, $A_x\subseteq S_x$ and any clique in a vortex of $\torso{G}{B_x}$ is contained in at most two parts in $\PP_x$. As such, for every child-adhesion clique $C$ of $\torso{G}{B_x}$, $C\setminus S_x$ is contained in at most three parts of $\PP_x$. By \cref{CliqueSumNew}, $G$ has an $H$-partition with width at most 
$m:= \max\{(6k+3)\cdot9\sqrt{(k+1)n},(k+18k+1)\ceil{\sqrt{n}}\}\leq (6k+3)\cdot9\sqrt{(k+1)n}$, 
where $H$ contains a vertex $\alpha$ such that $\tw(H-\alpha)\leq 3$. It therefore follows from  \cref{ObsPartitionProduct} that $G$ is contained in $H\boxtimes K_{\floor{m}}$ where $\tw(H)\leq 4$.
\end{proof}

%%%%%%%%%%%%%%%%%%%%%%%%%%%%%%%%%%
\section{\Large Proof of \cref{K3t}}

This section proves \cref{K3t} for $K_{3,t}$-minor-free graphs, where we assume throughout that $t\geq 1$. 
%\robert{Note that several claims in this section are false for $t=0$ (e.g. $K_2$ is a trivial counter-example to \cref{extremal} when $t=0$). Shall we add the following: As \cref{K3t} holds trivially when $t=0$, we assume throughout this section that $t>0$.}  \david{I have added ``where we assume throughout that $t\geq 1$'''} 
We use the following extremal function for $K^*_{3,t}$-minor-free graphs by \citet{KP10}. Here $K^*_{3,t}$ is the graph obtained from $K_{3,t}$ by adding an edge between each pair of vertices in the side of the bipartition with three vertices.

\begin{lem}[\citep{KP10}]
\label{extremal}
Every $K^*_{3,t}$-minor-free graph $G$ satisfies $|E(G)|\leq \alpha t \,|V(G)|$, for some constant $\alpha\geq1$.
\end{lem}

The following notation will be useful in the proof of \cref{K3t}. For a graph $G$, an induced subgraph $C$ of $G$, and sets $X,Y\subseteq V(G)$ such that $X,Y,V(C)$ are pairwise disjoint, let \defn{$\kappa_G(X,C,Y)$} be the maximum number of vertex-disjoint paths in $C$, each with an endpoint in $N_G(X)\cap C$ and an endpoint in  $N_G(Y)\cap C$. By Menger's Theorem there is a set $S\subseteq V(C)$ of size $\kappa_G(X,C,Y)$ separating $N_G(X) \cap C$ and $N_G(Y)\cap C$ in $C$. If $X=\{x\}$ then replace $X$ by $x$ in this notation, and similarly for $Y$. 

The following lemma is the key to the proof of \cref{K3t}. Here $\alpha$ is from  \cref{extremal}.  
\begin{lem}
\label{new}
Let $G$ be a $K^*_{3,t}$-minor-free graph on $n$ vertices. Let $X$ and $Y$ be disjoint non-empty sets of vertices in $G$ such that $G[X]$,  $G[Y]$ and $G[X\cup Y]$ are connected. 
Then there is a set $S\subseteq V(G-X-Y)$ such that:
\begin{itemize}
\item $|N_G[S]| \leq t\sqrt{3\alpha n}$, 
\item $\kappa_G(X,C,Y) \leq t\sqrt{3\alpha n}$ for every component $C$ of $G-X-Y-S$, and
\item $G[X\cup S]$ and $G[Y\cup S]$ are connected.
\end{itemize}
\end{lem}

\begin{proof}
Let $Q_1,\dots,Q_m$ be a maximum-size set of vertex-disjoint paths in $G-X-Y$ between $N_G(X)\setminus Y$ and $N_G(Y)\setminus X$. 
If $m=0$, then the lemma holds trivially with $S=\varnothing$, thus we may assume $m\geq 1$. 
Define $J$ to be the auxiliary graph with vertex set $\{q_1,\dots,q_m\}$ where $q_iq_j\in E(J)$ whenever there is a path in $G-X-Y$ joining $Q_i$ and $Q_j$, and avoiding each $Q_\ell$ with $\ell\not\in\{i,j\}$. 

Consider a component $J'$ of $J$. 
Let $(V_0,V_1,\dots)$ be a \textsc{BFS}-layering of $J'$. 
So $|V_0|=1$. We claim that $|V_i|< t$ for each $i\geq 1$. Suppose for the sake of contradiction that $|V_i|\geq t$ for some $i\geq 1$. Without loss of generality, $q_1,\dots,q_t\in V_i$. Let $A$ be the union of:
(1) all paths $Q_j$ corresponding to vertices in $V_0\cup\dots\cup V_{i-1}$,
(2) all paths in $G-X-Y$ corresponding to edges in $J[V_0\cup\dots\cup V_{i-1}]$, and 
(3) all paths in $G-X-Y$ corresponding to edges in $J$ between a vertex in $V_{i-1}$ and $q_1,\dots,q_t$, not including the vertex in $Q_1\cup\dots\cup Q_t$.
By construction, $A$ is a connected subgraph of $G-X-Y$ disjoint from $Q_1\cup\dots\cup Q_t$ and adjacent to each of $Q_1,\dots,Q_t$. Each of $A,Q_1,\dots,Q_t$ intersect $N_G(X)$ and $N_G(Y)$. Thus $X,Y,A,Q_1,\dots,Q_t$ form a $K^*_{3,t}$-model in $G$. This contradiction shows that $|V_i|< t$ for each $i\geq 0$. 

Concatenate the above-mentioned layerings of each component of $J$ to obtain a layering $(V_0,V_1,\dots)$ of $J$ with $|V_i|<t$ for each $i$. 
Assign each vertex $q_j$ in $J$ a weight of  $|N_G[Q_j]|$. 
The total weight is at most $|V(G)|+2|E(G)|$, 
which by \cref{extremal} is at most 
$(2\alpha t+1)n
\leq 3\alpha t n$ since  $t\geq 1$. 
Weight each set $V_i$ by the total weight of the vertices in $V_i$. 
Let $p:=\ceil{\sqrt{ 3\alpha n}}$.
There exists $i\in\{0,\dots,p-1\}$ such that 
$Z:=\bigcup\{V_j:j\equiv i\pmod{p}\}$ has weight at most 
 $3\alpha tn/p\leq t\sqrt{3\alpha n}$, and each component of $J-Z$ has less than $(p-1)t\leq t\sqrt{3\alpha  n}$ vertices. 
Let $S:=\bigcup\{Q_i:q_i\in Z\}$. 
By construction, $|N_G[S]|$ is at most the weight of $Z$, which is at most $t\sqrt{3\alpha n}$. Moreover, since $G[X\cup Q_i]$ and $G[Y\cup Q_i]$ are connected for all $i\in [m]$, it follows that $G[X\cup S]$ and $G[Y\cup S]$ are connected.

Consider a component $C$ of $G-X-Y-S$. 
Since each component of $J-Z$ has at most $t\sqrt{3\alpha n}$ vertices, 
the number of paths $Q_i$ that pass through $C$ is at most $t\sqrt{3\alpha n}$. By the choice of $Q_1,\dots,Q_m$, we have $\kappa_G(X,C,Y) \leq t \sqrt{3\alpha n}$. 
%\gwen{Same question as before: What about edges between $X$ and $Y$?} \david{This is fixed by the changes to the  definition of $\kappa$, right?}
\end{proof}

\cref{K3t} follows from \cref{ObsPartitionProduct} and the next lemma. 

\begin{lem}
\label{NewK3t}
Let $G$ be a $K^*_{3,t}$-minor-free graph on $n$ vertices. Let $Q$ be a clique in $G$ with $|Q|\leq 2$ such that if $Q=\{x,y\}$ with $x\neq y$, then $\kappa_G(x,C,y)\leq 2t\sqrt{3\alpha n}$ for every component $C$ of $G-x-y$, where $\alpha$ is from \cref{extremal}. Then $G$ has a partition $\PP$ with non-empty parts, with width at most $\leq 4t\sqrt{3\alpha n}$, with $\tw(G/\PP)\leq 2$, and with  $\{v\}\in\PP$ for each $v\in Q$.
\end{lem}

\begin{proof}
We proceed by induction on $|V(G)\setminus Q|$.
The result is trivial if $V(G)=Q$. Now assume that $V(G)\neq Q$. If $Q=\emptyset$ then the result follows by induction where $Q:=\{v\}$ and $v$ is any vertex in $G$. Now assume that  $Q\neq\emptyset$. If $G$ is disconnected, then the result follows by applying induction in each component $C$ of $G$ with the clique $Q\cap V(C)$. Now assume that $G$ is connected.
First consider the case in which $Q=\{x\}$. Since $G$ is connected and $V(G)\neq Q$, there is a neighbour $v$ of $x$. 
By \cref{new} applied to $(G,\{x\},\{v\})$, 
there is a set $S \subseteq V(G-x-v)$ such that:
\begin{itemize}
\item $|N_{G}[S]|\leq t\sqrt{3\alpha n}$, 
\item $\kappa_G(x,C,v)\leq t\sqrt{3\alpha n}$ for each component $C$ of $G-x-v-S$, and 
\item $G[S\cup\{v\}]$ is connected. 
\end{itemize}
Let $G'$ be obtained from $G$ by contracting $S\cup\{v\}$ into a single vertex $v'$. So $G'$ is $K^*_{3,t}$-minor-free and $xv'$ is an edge of $G'$.
For each component $C'$ of $G'-x-v'$, we have 
\(\kappa_{G'}(x,C',v') \le
\kappa_{G}(x,C',v) + |N_{G}[S]| \le 2t\sqrt{3\alpha n}\).
Apply induction to $G'$ and $Q':=\{x,v'\}$ 
to obtain a partition $\PP'$ of $G'$
of width at most $4t\sqrt{3\alpha n}$ such that $\tw(G'/\PP')\leq 2$ and $\{x\},\{v'\}\in\PP'$.
Let $\PP$ be the partition of $G$ obtained from $\PP'$ by replacing $\{v'\}$ by $S\cup\{v\}$. 
So $\PP$ has width at most $\max\{4t\sqrt{3\alpha n},|S|+1\}=4t\sqrt{3\alpha n}$ and $\{x\}\in \PP$. Since $G/\PP\cong G'/\PP'$
we have $\tw(G/\PP)\leq 2$. 

Now consider the case in which $|Q|=2$ and $Q=\{x,y\}$. 

First, suppose that no component of $G-x-y$ intersects $N_G(x)$ and $N_G(y)$. Let $G_x$ be the subgraph of $G$ induced by $\{x\}$ and the components of $G-x-y$ that intersect $N_G(x)$. Let $G_y$ be the subgraph of $G$ induced by $\{y\}$ and the components of $G-x-y$ that intersect $N_G(y)$. By induction, $G_x$ has a partition $\PP_x$ of width at most $4t\sqrt{3\alpha n}$ such that $\tw(G_x/\PP_x)\leq 2$ and $\{x\}\in\PP_x$. Similarly, $G_y$ has a partition $\PP_y$ of width at most $4t\sqrt{3\alpha n}$ such that $\tw(G_y/\PP_y)\leq 2$ and $\{y\}\in\PP_y$. Let $\PP:=\PP_x\cup\PP_y$. So $\PP$ is a partition of $G$, and $G/\PP$ is obtained from the disjoint union of $G_x/\PP_x$ and $G_y/\PP_y$ by adding the edge $\{x\}\{y\}$. So $\tw(G/\PP)\leq 2$.

Now assume that some component $C$ of $G-x-y$ intersects both $N_G(x)$ and $N_G(y)$. 
By assumption, $\kappa_G(x,C,y)\leq 2t\sqrt{3\alpha n}$.
By Menger's theorem, there exists $S\subseteq V(C)$ such that $|S|\leq 2t\sqrt{3\alpha n}$ and $S$ separates  $N_G(x) \cap V(C)$ and $N_G(y)\cap V(C)$ in $C$. Choose $S$ to be minimal. 
Observe that $S\neq \varnothing$. 
No component of $C-S$ intersects both $N_G(x)$ and $N_G(y)$. Let $D_x$ be the union of the components of $C-S$ that intersect $N_G(x)$. Let $D_y$ be the union of the components of $C-S$ that intersect $N_G(y)$. Let $F$ be the union of the components of $C-S$ that intersect neither $N_G(x)$ nor $N_G(y)$. 
Let $G_C := G[(V(C)\cup\{x,y\})\setminus V(F)]$. 

Let $Y:=\{y\}\cup V(D_y)\cup S$. By the minimality of $S$, $G[Y]$ is connected, and $G[\{x\}\cup Y]$ is connected since $xy\in E(G)$. By \cref{new} applied to $(G_C,\{x\},Y)$, there is a set 
$S_x\subseteq V(G_C-x-Y)=V(D_x)$ such that:
\begin{itemize}
\item $|N_{G_C}[S_x]| \leq t\sqrt{3\alpha n}$, 
%\item $G[S_x\cup \{x\}]$ is connected, 
\item $\kappa_{G_C}(x,C',Y) \leq t\sqrt{3\alpha n}$ for every component $C'$ of $D_x-S_x$, and 
\item $G_C[S_x\cup Y]$ is connected. 
\end{itemize}
Let $G_x$ be the graph obtained from $G_C$ by contracting $S_x\cup Y$ into a single vertex $z$.
Thus $G_x$ is $K^*_{3,t}$-minor-free and $xz$ is an edge of $G_x$. Consider a component $C'$ of $G_x-x-z$. Then $C'$ is a component of $D_x-S_x$, and 
\[\kappa_{G_x}(x,C',z)
=\kappa_{G_C}(x,C',S_x\cup Y)
\leq \kappa_{G_C}(x,C',Y) + |N_{G_C}[S_x]|
\leq 2t\sqrt{3\alpha n}.\] 
%\gwen{Above, shouldn't $S\cup S_x$ be $S_x \cup Y$?} \david{yes, changed now}
By induction, $G_x$ has a partition $\PP_x$ of width at most $4t\sqrt{3\alpha n}$ such that $\tw(G_x/\PP_x)\leq 2$ and $\{x\},\{z\}\in \PP_x$. 
(Note that $|V(G_x)| < |V(G)|$ since $S\neq \varnothing$, so we may apply induction.)

Let $X:=\{x\}\cup V(D_x)\cup S$. By a symmetric argument to the above, there is a set $S_y\subseteq V(D_y)$ such that:
\begin{itemize}
\item $|N_{G_C}[S_y]| \leq t\sqrt{3\alpha n}$, 
%\item $G_C[S_y\cup \{y\}]$ is connected, and
\item $\kappa_{G_C}(y,C',X) \leq t\sqrt{3\alpha n}$ for every component $C'$ of $D_y-S_y$, and 
\item $G_C[S_y\cup X]$ is connected.
\end{itemize}
Let $G_y$ be the graph obtained from $G_C$ by contracting $S_y\cup X$ into a single vertex $z$.
Thus $G_y$ is $K^*_{3,t}$-minor-free and $yz$ is an edge of $G_y$. By a symmetric argument, 
$G_y$ has a partition $\PP_y$ of width at most $4t\sqrt{3\alpha n}$ such that $\tw(G_y/\PP_y)\leq 2$ and $\{y\},\{z\}\in \PP_y$.

Note that $G[X\cup Y]$ is connected. 
Let $G_F$ be the graph obtained from $G[\{x,y\}\cup V(C)]$ by contracting $X\cup Y$ into a single vertex $z$. So $V(G_F)=\{z\}\cup V(F)$, and $G_F$ is $K^*_{3,t}$-minor-free. By induction, $G_F$ has a partition $\PP_F$ of width at most $4t\sqrt{3\alpha n}$ such that $\tw(G_F/\PP_F)\leq 2$ and $\{z\}\in\PP_F$. 

Let $G':=G-V(C)$. So $G'$ is $K^*_{3,t}$-minor-free, and $xy$ is an edge of $G'$. By induction, $G'$ has a partition $\PP'$ of width at most $4t\sqrt{3\alpha n}$ such that $\tw(G'/\PP')\leq 2$ and $\{x\},\{y\}\in\PP'$. 

Let $\PP$ be the partition of $G$ obtained from $\PP_x\cup\PP_y\cup \PP_F \cup\PP'$ by replacing each of the three instances of $\{z\}$ by $S\cup S_x\cup S_y$. The width of $\PP$ is at most $4t\sqrt{3\alpha n}$. Note that $G/\PP$ is obtained by pasting the four graphs $G_x/\PP_x$, $G_y/\PP_y$, $G_F/\PP_F$ and $G'/\PP'$ on the triangle $\{x\},\{y\},S\cup S_x\cup S_y$, where each of the four graphs contains vertices in two of $\{x\}$, $\{y\}$ and $S\cup S_x\cup S_y$. Thus  $G/\PP$ is obtained from graphs of treewidth at most 2 by pasting on edges. Hence $\tw(G/\PP)\leq 2$ and $\{x\},\{y\}\in \PP$. 
\end{proof}

\section{\Large Open Problems}

It is an intriguing open problem to determine $f(\GG)$ for a given proper minor-class $\GG$. It is possible that $f(\GG)\leq 2$ for every minor-closed class $\GG$. This is open even when $\GG$ is the class of $K_5$-minor-free graphs \cite{ISW}. Let $\mathcal{A}$ be the class of apex graphs\footnote{A graph $H$ is \defn{apex} if $H-v$ is planar for some vertex $v$ of $H$.}, which is minor-closed. It is open whether $f(\mathcal{A})\leq 2$. This is equivalent to the following open problem (which would strengthen \cref{Planar}): for every $n$-vertex planar graph $G$, does there exist an apex-forest\footnote{A graph $H$ is an \defn{apex forest} if $H-v$ is a forest for some vertex $v$ of $H$.} $H$ such that $G$ is contained in $H\boxtimes K_m$ where $m\in\OO(\sqrt{n})$? 

It is also open whether treewidth can be replaced by pathwidth in  \cref{ConstantTwSqrtn,Planar,K3t}. That is, for a proper minor-closed class $\GG$, are there integers $k,c$ such that every $n$-vertex graph in $\GG$ is contained in $H\boxtimes K_m$, for some graph $H$ with pathwidth at most $k$, where $m\leq c\sqrt{n}$? Two pieces of evidence suggest a positive answer. First, $n$-vertex graphs in a proper minor-closed class have pathwidth $\OO(\sqrt{n})$; see \citep{Bodlaender98}. Second, if $\GG$ has bounded treewidth, then the answer is `yes' with $k=1$, since \citet{DvoWoo} showed that $n$-vertex graphs in $\GG$ have $H$-partitions of width $\OO(\sqrt{n})$ where $H$ is a star, which has pathwidth~1. This question is open for planar graphs. 
 
Analogous questions are interesting and open for several non-minor-closed classes~\citep{DvoWoo}.

\paragraph{Acknowledgement:} This work was partially completed at the 10th Annual Workshop on Geometry and Graphs held at Bellairs Research Institute in February 2023.

{\fontsize{11pt}{12pt}\selectfont
\bibliographystyle{DavidNatbibStyle}
\bibliography{DavidBibliography}}

\end{document}